\documentclass[12pt,reqno]{amsart}
\usepackage[utf8]{inputenc}
\usepackage[T1]{fontenc}
\usepackage[usenames, dvipsnames]{color}
\usepackage{ulem}

\usepackage{dsfont, amsfonts, amsmath, amssymb,amscd, stmaryrd, latexsym, amsthm, dsfont}
\usepackage[frenchb,english]{babel}
\usepackage{enumerate}
\usepackage{longtable}
\usepackage{geometry}
\usepackage{float}
\usepackage{tikz}
\usetikzlibrary{shapes,arrows}
\usepackage{float}
\usepackage{tikz}
\usepackage{xypic}
\usetikzlibrary{shapes,arrows}

\theoremstyle{plain}

\newtheorem{theorem}{Theorem}[section]
\newtheorem{lemma}[theorem]{Lemma}
\newtheorem{proposition}[theorem]{Proposition}

\theoremstyle{remark}
\newtheorem{remark}[theorem]{\bf Remark}

\def\QQ{\mathbb{Q}}

\usepackage{hyperref}
\hypersetup{
	colorlinks=true,
	urlcolor=blue,
	citecolor=blue}

\begin{document}

	\selectlanguage{english}
	\title[The construction of the Hilbert genus fields...]{The construction of the Hilbert genus fields of real cyclic quartic fields}

	\author[M. M. Chems-Eddin]{M. M. Chems-Eddin}
	\address{Mohamed Mahmoud Chems-Eddin: Mohammed First University, Mathematics Department, Sciences Faculty, Oujda, Morocco }
	\email{2m.chemseddin@gmail.com}
	
	\author[M. A. Hajjami]{M. A. Hajjami} 
	\address{Moulay Ahmed Hajjami, Department of  Mathematics, Faculty of  Sciences and Technology, Moulay Ismail University of Meknes, Errachidia, Morocco.}
	\email{a.hajjami76@gmail.com}

	\author[M. Taous]{M. Taous}
	\address{Mohammed Taous, Department of  Mathematics, Faculty of  Sciences and Technology, Moulay Ismail University of Meknes, Errachidia, Morocco.}
	\email{taousm@hotmail.com}

	\keywords{Real cyclic quartic  fields, unramified extensions,  Hilbert genus fields.}
	\subjclass[2010]{11R16, 11R29, 11R27, 11R04, 11R37}
	
	\begin{abstract}
		Let $p$ be a prime number such that $p=2$ or $p\equiv 1\pmod 4$. Let $\varepsilon_p$ denote the  fundamental unit of  $\mathbb{Q}(\sqrt{p})$ and let $a$ be a positive square-free integer. In the present  paper, we construct the Hilbert genus field of the real cyclic quartic    fields  $\mathbb{Q}(\sqrt{a\varepsilon_p\sqrt{p}})$.
	\end{abstract}
	
	\selectlanguage{english}
	
	\maketitle
	
	\section{\bf Introduction}\label{sec:1}
	Let $k$ be a number field  and let     $H(k)$  denote  the Hilbert class field of $k$, that   is
	the maximal abelian unramified extension of $k$. It is known by class field theory that  the Galois group    of the extension  $H(k)/k$, i.e., $\mathrm{G}:=\mathrm{Gal}(H(k)/k)$,  is isomorphic to $\mathbf{C}l(k)$, the class group of $k$ (cf. \cite[p. 228]{ref12}).

	The       Hilbert genus field of $k$,   denoted by   $E(k)$,    is the invariant field
	of $\mathrm{G}^2$. Thus, by Galois theory, we  have:
	$$\mathbf{C}l(k)/\mathbf{C}l(k)^2  \simeq \mathrm{G}/\mathrm{G}^2 \simeq \mathrm{Gal}(E(k)/k),$$
	and therefore, $2$-rank $(\mathbf{C}l(k))$ = $2$-rank $(\mathrm{Gal}(E(k)/k))$. On the other hand, $E(k)/k$ is the maximal unramified Kummer extension of exponent $2$. Thus, by Kummer theory (cf. \cite[p. 14]{ref16}), there exists a unique multiplicative group $\Delta(k)$ such that
	$$
	E(k)=H(k)\cap k(\sqrt{k^*})= k(\sqrt{\Delta(k)})  \text{  and }  {k^*}^2 \subset {\Delta(k)} \subset k^*.
	$$
	Therefore, the construction of  the Hilbert genus field of $k$ is  equivalent  to give a set of generators for the finite group $\Delta(k)/k{^*}^2$.\\
	Note that many    mathematicians showed their interest to this problem by several  studies on many different number fields.  For instance,   Bae and Yue   studied the Hilbert genus field of the fields $\mathbb{Q}(\sqrt{p}, \sqrt{d})$, for 
	a prime number $p$ such that, $p=2$ or $p\equiv 1\pmod 4$, and    a positive square-free  integer $d$ (cf. \cite{ref2}).

	After that, Ouang and Zhang have determined the Hilbert genus field of the imaginary biquadratic fields $ \mathbb{Q}(\sqrt{\delta}, \sqrt{d})$, where $\delta =-1, -2$  or $-p$ with $p\equiv 3\pmod 4$ a prime number and $d$ any square-free integer. Thereafter,  they   constructed   the Hilbert genus field of real biquadratic fields $\mathbb{Q}(\sqrt{\delta}, \sqrt{d})$, 
	for   any  positive square-free integer $d$, and   $\delta=p, 2p$ or $p_1p_2$ where $p$, $p_1$ and $p_2$ are prime numbers congruent to $3\pmod 4$, such that the class number of $ \mathbb{Q}(\sqrt{\delta})$ is odd (cf.  \cite{ref17,ref18}). 
	Recently, in \cite{hajjaChems}, the first author named and the second author named constructed the Hilbert genus field    of the  imaginary cyclic quartic fields    $\mathbb{Q}(\sqrt{-d{\varepsilon}_p \sqrt{p}})$, where $p$ is a prime number   such that, $p=2$ or $p\equiv 1\pmod 4$, and  $d$ is   a positive square-free  integer.
	For more papers on this problem, we refer the reader to \cite{ref20,ref19,ref7}.
	
	The construction of  the Hilbert genus fields of  real cyclic quartic fields  demands particular preparations  and very long investigations, therefore, 
	in the present work,  we shall construct the Hilbert genus fields of  real cyclic quartic fields of the form  $K=\mathbb{Q}(\sqrt{a{\varepsilon}_p \sqrt{p}})$, for a prime number $p$ such that $p=2$ or $p\equiv 1\pmod 4$, and a positive square-free integer $a$ relatively prime to $p$.

	Note  therefore that   this paper completes the previous studies of  the construction of the Hilbert genus fields of all abelian quartic fields  that have quadratic subfields whose class numbers are odd.
	
	The plan of this paper is the following: In \S \ref{sec:2}, we shall collect some preliminary results that we shall use for the next. In \S \ref{sec:3}, we prove further preliminary results on Diophantine equations.
	In \S \ref{sec:4}, we  investigate   the genus fields of the real cyclic quartic fields $K$. In \S \ref{sec:5}, we shall construct the Hilbert genus fields of the fields $K$.

	\begin{center}
		{	\large \bf{Notations}}
	\end{center}
	Let $k$ be a number field. Throughout this paper, we shall respect the following notations:
	\begin{enumerate}[$\bullet$]
		\item $\mathbf{C}l(k)$: the class group of $k$,
		\item $\mathcal{O}_k$: the ring of integers of  $k$,
		\item $N_{k/k'}$:  the norm map of an extension $k/k'$,
		\item $E_k$: the unit group of $k$,
		\item $k^*$: the nonzero elements of  $k$,
		\item $k_{(*)}$: the  narrow genus field of  $k$,
		\item $k^{(*)}$: the   genus field of  $k$,
		\item $E(k) $:  the Hilbert genus field of $k$,
		\item $H(k) $:  the Hilbert  class field of $k$,
		\item $\delta_k$: the absolute discriminant of $k$,
		\item $\delta_{k/k'}$: the generator of the relative discriminant of an extension $k/k'$,
		\item $r_{2}(A)$:  the $2$-rank of a finite  abelian group $A$,
		\item  $\varepsilon_{d} $:  the fundamental unit of $\QQ(\sqrt{d})$, where $d$ is a positive square-free integer,
		\item $p$: a prime number such that $p=2$ or $p\equiv 1\pmod 4$,
		\item $a$: a positive square-free  integer relatively prime to $p$,
		\item $\delta=a\varepsilon_p \sqrt{p}$,
		\item $k_0=\mathbb{Q}(\sqrt{p})$,
		\item $K=k_0(\sqrt \delta)$: an imaginary quartic cyclic number field,
		\item $q_j$: an odd prime integer,
		\item $i_0$ and $j_0$: two positive integers defined in pages \pageref{io} and \pageref{j0} respectively, 
		\item $\left(\dfrac{\cdot}{\cdot}\right)$:  the Legendre symbol,
		\item $\left(\dfrac{\cdot}{\cdot}\right)_4$:   the rational biquadratic residue symbol.		
		\item  Let $A=\{a_i\}_{i\in I}$ be a finite set of complex numbers. To simply notations in some places throughout the paper we shall put $  k(\{a_i\}_{i\in I}):=k(a_1,...,a_s)$, i.e., the extension of $k$ generated by the elements of $A$.
	\end{enumerate}

	For further   notations see   the beginning of each section below.

	\section{\bf Preliminary results}\label{sec:2}
	In the present  section, we start our preparations by collecting some known results, of the theory of quartic cyclic fields and genus fields,  that we will need in what follows. 
	Let $L$ be a cyclic quartic extension of the rational number field $\mathbb{Q}$. It is known   that  $L$ can be expressed uniquely  in the form: 
	$L=\mathbb{Q}(\sqrt{a(d+b\sqrt{d})}),$
	for some integers $a$, $b$, $c$ and $d$ such that  $d=b^2+c^2$ is square-free with $b>0$ and $c>0$, and
	$a$ is an odd square-free integer relatively prime to $d$ (cf. \cite{ref6,ref21}).  Note that $L$ possesses a unique quadratic subfield $k=\mathbb{Q}(\sqrt{d})$. The following lemmas are used to compute ramifications indexes of prime numbers that ramify in $K$.

	\begin{lemma}[\cite{ref9}]\label{lemma 2.1; discriminants}
		Keep the above notations. We have:
		\begin{enumerate}[\rm 1.]
			\item  	The absolute discriminant of $L$ is given by $\delta_L$, where:
			$$
			\delta_L=\left\{\begin{array}{lll}
				2^8a^2d^3, & \text { if } & d \equiv 0\pmod 2, \\
				2^4a^2d^3, & \text { if } & d \equiv 1\pmod 4, b \equiv 0\pmod 2, a+b \equiv 3\pmod 4, \\ 
				a^2d^3,& \text { if } & d \equiv 1\pmod 4, b \equiv 0\pmod 2, a+b \equiv 1\pmod 4, \\ 
				2^6a^2d^3, & \text { if } & d \equiv 1\pmod 4, b \equiv 1\pmod 2.\end{array}\right.
			$$ 
			\item The relative discriminant of  $L/k$ is given by  $\Delta_{L/k}=\delta_{L/k}\mathcal{O}_k,$ where:
			$$
			\delta_{L/k}=\left\{\begin{array}{lll}
				4a \sqrt{d}, & \text { if } & d \equiv 0\pmod 2,  \\
				4a \sqrt{d}, & \text { if } & d \equiv 1\pmod 4, b \equiv 0\pmod 2, a+b \equiv 3\pmod 4, \\ 
				a\sqrt{d},& \text { if } & d \equiv 1\pmod 4, b \equiv 0\pmod 2, a+b \equiv 1\pmod 4, \\ 
				8a\sqrt{d}, & \text { if } & d \equiv 1\pmod 4, b \equiv 1\pmod 2.\end{array}\right.
			$$ 
		\end{enumerate}
	\end{lemma}

	\begin{lemma}[\cite{ref9}]\label{lm 2.2}
		Keep the above notations. If the class number of $k=\mathbb{Q}(\sqrt{d})$ is odd, then $L=\mathbb{Q}(\sqrt{a' \varepsilon_d\sqrt{d}}),$ where 
		$$
		a'=
		\begin{cases} 2a, &\text{ if } d\equiv 1\pmod 4 \text{ and } b\equiv 1\pmod 2,\\
			a, &\text{otherwise}.\end{cases}
		$$
	\end{lemma}

	In \cite{frohlich}, Fr\"ohlich defined the genus field of a   number
	field $k$ as the maximal extension of $k$
	which is unramified at all finite primes
	of $k$ of the form $kk_1$,  where $k_1$ is an abelian extension of $\QQ$. Here, this field is called the narrow genus field of $k$ and it is
	denoted   by $k_{(*)}$. The genus field  of $k$ (denoted by $k^{(*)}$) is defined as the maximal
	extension of $k$ which is unramified at all finite and infinite primes of $k$ of the
	form $kk_1$, where $k_1$ is an abelian extension of $\QQ$. Note that if $k$ is an  abelian number field, then $k^{(*)}=k_{(*)}$ if $k$ is imaginary and $ k^{(*)}$ is the maximal real sub-extension of $k_{(*)}$ (i.e. $k_{(*)}^+$)  if $k$ is real.	We have:

	\begin{proposition}[\cite{ref11}]\label{prop 2.4}
		Let $L$ be an abelian extension of $\mathbb{Q}$ of degree $n$. 
		If $n=r^s$, where $r$ is a prime number and $s$ is a positive integer, then 
		$$L_{(*)} =\left(\prod_{p/\delta_L, p \neq r }M_p\right)L,$$
		where $M_p$ is the unique subfield of degree $e_p$ (the ramification index of $p$ in $L$) over $\mathbb{Q}$ of $\mathbb{Q}(\zeta_p)$ the $p$-th cyclotomic field and $\delta_L$ is the discriminant of $L$. 
	\end{proposition}

	\begin{proposition}[\cite{ref15}, p. 160]\label{prop 2.5}
		If $p$ is a prime number such that $p\equiv 1\pmod 4$, then $\mathbb{Q}(\sqrt{\varepsilon_p^*\sqrt{p}})$ is the quartic subfield of $\mathbb{Q}(\zeta_p)$, where
		$\varepsilon_p^*=\left(\frac{2}{p}\right)\varepsilon_p$ and $\mathbb{Q}(\sqrt{2+\sqrt{2}})$ is the real quartic field of $\mathbb{Q}(\zeta_{16})$.	 
	\end{proposition}

	We close this section with the following two results.

	\begin{lemma}\label{lemma 4.1}
		Let $k/k'$ be a quadratic extension of number fields such that the class number of $k'$ is odd.   Let $\Delta(k)$ denote the multiplicative group  such that ${k^*}^2 \subset {\Delta(k)} \subset k^*$ and $k(\sqrt{\Delta(k)})$ is the Hilbert genus field of $k$  $($\textnormal{cf}.  \S \ref{sec:1}$)$. Then 
		$$r_2(\Delta(k)/{k^*}^2)=r_2({\mathbf{C}l(k)})=t-e-1,$$
		where  $t$ is the number of  ramified primes (finite or infinite) in the extension  $k/k'$ and $e$ is  defined by   $2^{e}=[E_{k'}:E_{k'} \cap N_{k/k'}(k^*)]$.
	\end{lemma}
	\begin{proof} This follows by well known ambiguous class number formula (cf.  \cite{Gr}) and Lemma \cite[Lamma 4.1]{hajjaChems}.
	\end{proof}

	\begin{proposition}[\cite{ref8}]\label{unramifed quad ex}
		Let $k/k'$ be a quadratic extension of number fields  and $\mu$ a number of $k'$, coprime with $2$, such that $k=k'(\sqrt{\alpha})$. The extension $k/k'$ is unramified  at all finite primes of $k'$ if  and only  if     the two following items  hold:
		\begin{enumerate}[\indent\rm 1.]
			\item the ideal generated by $\alpha$ is the square of the fractional ideal of $k'$, and 
			\item there exists a nonzero number $\xi$ of $k'$ verifying $\alpha \equiv \xi^2\pmod{4}$.
		\end{enumerate}
	\end{proposition}

	Let us close this section by recalling the definition and some properties of  the rational biquadratic residue symbol. 
	Let $p$ be a prime number and  $a\in \mathbb{Z}\backslash \mathbb{Z}$ a quadratic residue modulo $p$. Then  we define  the \textit{rational biquadratic residue symbole modulo }$p$ by 
	\begin{eqnarray*}
		\genfrac(){}{0}{a}{p}_4=\left\{  \begin{array}{ll}
			1& \text{ if } a \text{ is a biquadratic residue modulo} p\\
			-1& \text{ otherwise.}
		\end{array} \right.
	\end{eqnarray*}
	We have :\\
	$\noindent$ $\bullet$ If $p\equiv 1\pmod 4$, then 
	$$\genfrac(){}{0}{a}{p}_4=\pm1\equiv a^{\frac{p-1}{4}}\pmod p.$$

	$\noindent$ $\bullet$ If now $a$ is an integer $a\equiv 1\pmod 8$, the symbol $\genfrac(){}{0}{a}{2}_4$ is defined by
	$$\genfrac(){}{0}{a}{2}_4=1 \text{ if }  a\equiv 1\pmod{16},\;\;\;\;\;\;\;\;\; \genfrac(){}{0}{a}{2}_4=-1 \text{ if }  a\equiv 9\pmod{16}.$$
	
	It turns out that   $\genfrac(){}{0}{a}{2}_4=(-1)^{\frac{a-1}{8}}$. For more details, we refer the reader to \cite[Chapter 7]{koch}.
	\section{\bf Further preliminaries on Diophantine equations}\label{sec:3}	
	In this section, we prove the next propositions that help  to explicit  conditions in several cases in our main theorems.
		Let  $p$ and $q$ be two prime numbers in one of the next three cases: 
		\begin{eqnarray}
		 &p&\equiv q \equiv1\pmod 4 	\text{ and }   \left(  \frac{p}{q}\right)=1 \label{form1}\\
		&p&=2  \text{ and }  q \equiv1\pmod 8 \label{form2}\\
		 &p&\equiv -q \equiv1\pmod 4  	\text{ and }   \left(  \frac{p}{q}\right)=1 \label{form3}
	\end{eqnarray}
 By  \cite[Lemma 4.3]{hajjaChems} and \cite{leowillia}, there exist two natural integers $x$ and $y$ satisfying the next Diophantine equation:
	\begin{eqnarray}\label{eq dioph}
		x^2-py^2=q^{\lambda h},
	\end{eqnarray}
	where $h$ is the class number of  $k_0=\mathbb{Q}(\sqrt{p})$ and $\lambda=\begin{cases}1,& \text{ if } p=2 \text{ or } p\equiv 1\pmod 8\\
		3, & \text{ if }  p\equiv 5\pmod 8.\end{cases}$\\
 	and we have 
		\begin{enumerate}[\indent\rm 1.]
		\item If $p$ and $q$ satisfy \eqref{form1} , then $x$ is odd and  $y$ is even.
		\item  If $p$ and $q$ satisfy \eqref{form2}, then  $x$ is odd and $y$ is divisible by $4$.
		\item  If $p$ and $q$ satisfy \eqref{form3}, then $x$ is even and  $y$ is odd.
	\end{enumerate}
	In the next proposition, we give more precise information on these solutions.
	\begin{proposition}\label{prop 3.1}
		Assume that the primes  $p$ and $q$ are odd. We have
	  	\begin{enumerate}[\rm 1.]
		\item If   $p\equiv q\equiv 1\pmod 4$ and   $y=2^ey'$, with   $e\geq  1$ and  $y'$ is odd, then $e=1$  if and only if $q\equiv 5\pmod 8$. In all cases we have 
		\begin{eqnarray}\label{taous1}
	  \left(  \frac{2}{q}\right)^e=\left(  \frac{2}{q}\right)=(-1)^{\frac{y}{2}}.
		\end{eqnarray}
		\item If   $p\equiv 1\pmod 4$, $q\equiv 3\pmod 4$ and   $x=2^ex'$, with  $e\geq  1$ and $x'$ is odd, then $e=1$ if and only if  $pq\equiv 3\pmod 8$. In all cases we have 
		\begin{eqnarray}\label{taous2}
		\left(  \frac{2}{pq}\right)^e=\left(  \frac{2}{pq}\right)=(-1)^{\frac{x}{2}}.
		\end{eqnarray}       
	\end{enumerate} 
	\end{proposition}
\begin{proof}Firstly note that by easy calculus modulo $8$,  one can check that for any odd natural integer $m$, we have $m ^2\equiv 1\pmod 8$.
 	\begin{enumerate}[\indent\rm 1.]
		\item Assume that  $p\equiv q\equiv 1\pmod 4$. Since
	$x^2-py^2=q\cdot (q^2)^{\frac{\lambda h-1}{2}}$ and  $y=2^ey'$ and $e\geq  1$ with $y'$ is odd, this implies that $q\equiv 1+2^{2e}p\pmod 8$ (because $x$, $y'$ and $q$ are odd). Which implies that  $\frac{q-1}{4}\equiv 2^{2(e-1)}\pmod 2$ and
	$$
	 \left(  \frac{2}{q}\right)=(-1)^{\frac{q-1}{4}}=(-1)^{2^{2(e-1)}}.
	$$
	This shows that $e=1$ if and only if  $q\equiv 5\pmod 8$. In addition, we have $$\left(  \frac{2}{q}\right)^e=\left(  \frac{2}{q}\right)=(-1)^{\frac{y}{2}}.$$
	\item Now if  $p\equiv 1\pmod 4$ and  $q\equiv 3\pmod 4$. We similarly get that  $\frac{pq+1}{4}\equiv 2^{2(e-1)}\pmod 2$ and
	$$
	\left(  \frac{2}{pq}\right)=(-1)^{\frac{pq+1}{4}}=(-1)^{2^{2(e-1)}}.
	$$
	In other words, $e=1$ if and only if   $pq\equiv 5\pmod 8$. We also have  $$\left(  \frac{2}{pq}\right)^e=\left(  \frac{2}{pq}\right)=(-1)^{\frac{y}{2}}.$$
	 	\end{enumerate} 
 \end{proof}

	\begin{proposition}\label{lm symb buqia} 
		Let  $p$ and $q$ be in one of the three above cases.  
		\begin{enumerate}[\rm $\bullet$]
			\item Assume that  $p$ and $q$ satisfy  \eqref{form1} or  \eqref{form2}.  We have
			\begin{eqnarray*}\label{equiva}
				x+y\equiv1 \pmod 4 \Longleftrightarrow  \left(  \frac{p}{q}\right)_4=\left(  \frac{q}{p}\right)_4.
			\end{eqnarray*}
			\item Assume that  $p$ and $q$ satisfy  \eqref{form3}.
			\begin{enumerate}[\rm $\star$]
				\item If $\left(  \frac{x}{q}\right)=\left(  \frac{y}{q}\right)$, then
				\begin{eqnarray*}\label{equiva2}
					x+y\equiv1 \pmod 4 \Longleftrightarrow  \left(  \frac{q}{p}\right)_4=1.
				\end{eqnarray*}
				\item If $\left(  \frac{x}{q}\right)=-\left(  \frac{y}{q}\right)$, then
				\begin{eqnarray*}\label{equiva3}
					x+y\equiv1 \pmod 4 \Longleftrightarrow  \left(  \frac{q}{p}\right)_4=-1.
				\end{eqnarray*}
			\end{enumerate}

			
		\end{enumerate} 
	\end{proposition}
	\begin{proof}  Note  firstly that if $p$ and $q$ are two prime numbers  such that $p \equiv1\pmod 4$ and  $q$ is odd, with
		$\left(  \frac{p}{q}\right)=1$, then
		we have
		\begin{eqnarray}\label{eq b}
			\left(  \frac{x}{p}\right)=\left(  \frac{x^2}{p}\right)_4=\left(  \frac{py^2+q^{ \lambda h}}{p}\right)_4=
			\left(  \frac{q^{ \lambda h }}{p}\right)_4=\left(  \frac{q}{p}\right)_4 . 
		\end{eqnarray}  
		\begin{enumerate}[\rm $\bullet$]
			\item Now    assume that $p$ and $q$ satisfy \eqref{form1}. So we have
			
			\begin{eqnarray*}\left(  \frac{p}{x}\right)=\left(  \frac{py^2}{x}\right) =\left(  \frac{ x^2-q^{ \lambda h }}{x}\right) =
				\left(  \frac{-q^{ \lambda h }}{x}\right)=\left(  \frac{-1}{x}\right)\left(  \frac{q}{x}\right).
			\end{eqnarray*}
			
			Thus, by reciprocity  law
			\begin{eqnarray}\label{lmeq1}
				\left(  \frac{q}{p}\right)_4=\left(  \frac{-1}{x}\right)\left(  \frac{q}{x}\right).
			\end{eqnarray}
			On the other hand, we have 
			\begin{eqnarray*}
				\left(  \frac{x}{q}\right)=\left(  \frac{x^2}{q}\right)_4 =\left(  \frac{x^2-q^{ \lambda h }}{q}\right)_4 =\left(  \frac{py^2}{q}\right)_4=\left(  \frac{p}{q}\right)_4\left(  \frac{y^2}{q}\right)_4=\left(  \frac{p}{q}\right)_4\left(  \frac{y }{q}\right).	
			\end{eqnarray*}
			Thus, again by reciprocity  law, we have
			\begin{eqnarray*}\label{lmeqqs}
				\left(  \frac{q}{x}\right)=\left(  \frac{x}{q}\right)=
				\left(  \frac{p}{q}\right)_4\left(  \frac{y }{q}\right).
			\end{eqnarray*}

Let $y=2^ey'$ such that $e\geq 1$ and $y'$ is odd.  Thus, 
			\begin{equation}\label{eqd2}
				\left(  \frac{q }{x}\right)=\left(  \frac{p }{q}\right)_4 \left(  \frac{2^ey' }{q}\right)=\left(  \frac{p }{q}\right)_4 \left(  \frac{2  }{q}\right)^e\left(  \frac{ y' }{q}\right)=\left(  \frac{p }{q}\right)_4 \left(  \frac{2  }{q}\right)^e
			\end{equation}
			In fact 	$x^2-p2^{2e}y'^2=q^{\lambda h}$  and so $1=\left(  \frac{ q^{\lambda h} }{y'}\right)=\left(  \frac{ q }{y'}\right)=\left(  \frac{ y' }{q}\right)$. 
			It follows by \eqref{lmeq1} and \eqref{eqd2}  that  $\left(  \frac{q }{p}\right)_4=\left(  \frac{-1 }{x}\right)\left(  \frac{p }{q}\right)_4\left(  \frac{2 }{q}\right)^e$. According to the result \eqref{taous1}, we prove that: 
			\begin{equation*}
					\left(  \frac{-1 }{x+y}\right) =\left(  \frac{-1 }{x}\right)(-1)^{\frac{y}{2}} =\left(  \frac{p }{q}\right)_4\left(  \frac{q }{p}\right)_4\left(  \frac{2 }{q}\right)\left(  \frac{2 }{q}\right)=\left(  \frac{p }{q}\right)_4\left(  \frac{q }{p}\right)_4. 
			\end{equation*}
	Then, $$	x+y\equiv1 \pmod 4 \Longleftrightarrow  \left(  \frac{p}{q}\right)_4=\left(  \frac{q}{p}\right)_4.$$
	
\item  Assume that $p=2$ and $q \equiv1\pmod 8$. 	By \cite[Theorem 2]{kaplan76}, we  have
			$\left(  \frac{p}{2}\right)_4=\left(  \frac{2}{x}\right)$ and 
			$\left(  \frac{2}{p}\right)_4=\left(  \frac{-2}{x}\right)$.
			Thus, \begin{eqnarray*}
				\left(  \frac{p}{2}\right)_4= \left(  \frac{2}{p}\right)_4&\Longleftrightarrow&\left(  \frac{2}{x}\right)=\left(  \frac{-2}{x}\right)\\
				&\Longleftrightarrow&  \left(  \frac{-1}{x}\right)=1\\
				&\Longleftrightarrow& x \equiv1 \pmod 4\\
				&\Longleftrightarrow& x+y\equiv1 \pmod 4.
			\end{eqnarray*}

			\item  Assume that  $p$ and $q$ satisfy  \eqref{form3}. Let us start by collecting some relationships between the Legendre symbols. Since $x$ is even, put $x=2^ex'$, for a positive integer $e$ and an odd integer $x'$. Thus, we have:
			\begin{eqnarray*}
				\left(  \frac{x}{q}\right)=	\left(  \frac{2^ex'}{q}\right)&=&	\left(  \frac{2}{q}\right)^e	\left(  \frac{x'}{q}\right)=\left(  \frac{2}{q}\right)^e	\left(  \frac{-q}{x'}\right)
				=\left(  \frac{2}{q}\right)^e	\left(  \frac{-q^{\lambda h}}{x'}\right)\\
				&=&\left(  \frac{2}{q}\right)^e	\left(  \frac{x^2-q^{\lambda h}}{x'}\right)
				=\left(  \frac{2}{q}\right)^e	\left(  \frac{ py^2}{x'}\right)\\
				&=&\left(  \frac{2}{q}\right)^e	\left(  \frac{ p}{x'}\right)=\left(  \frac{2}{q}\right)^e	\left(  \frac{ x'}{p}\right).\\		  
			\end{eqnarray*}
			Thus, 
			$\left(  \frac{ x}{q}\right)=\left(  \frac{2}{q}\right)^e\left(  \frac{2}{p}\right)^e\left(  \frac{2^e x'}{p}\right)  $. Therefore we have
			\begin{eqnarray}\label{eq c}
			\left(  \frac{ x}{q}\right)=\left(  \frac{2}{q}\right)^e\left(  \frac{2}{p}\right)^e\left(  \frac{x}{p}\right)  
			\end{eqnarray}
			From equations  \eqref{eq b} et \eqref{eq c}, we get 
			\begin{eqnarray}\label{eq9e}
				\left(  \frac{x}{q}\right)=\left(  \frac{2}{q}\right)^e\left(  \frac{2}{p}\right)^e\left(  \frac{q}{p}\right)_4.
			\end{eqnarray}
			
			 On the other hand, notice that, since $q^h\equiv x^2 \pmod y$ and $h$ is odd, we have $			 \left(  \frac{q}{y}\right)=1.$ Therefore, by reciprocity law:
			 	\begin{eqnarray}\label{eq-1y}
			  \left(  \frac{y}{q}\right)=\left(  \frac{-1}{y}\right) 
			 \end{eqnarray}
 
			For the sake to  get  a characterization to $x+y\equiv 1\pmod 4$, in terms of $\left(  \frac{q}{p}\right)_4$, the   equations  \eqref{eq9e}, \eqref{eq-1y} and \eqref{taous2} imply the following equality:
				\begin{equation*}
			\left(  \frac{-1 }{x+y}\right) =\left(  \frac{-1 }{y}\right)(-1)^{\frac{x}{2}} =\left(  \frac{y}{q}\right)\left(  \frac{2}{pq}\right)=\left(  \frac{y}{q}\right)\left(\frac{x}{q}\right)\left( \frac{q }{p}\right)_4. 
			\end{equation*}
Which completes the proof. 
		\end{enumerate} 
	\end{proof}
	
	\begin{remark}
		Note that for the sake to simplify  the notations and the conditions used   in many places in our main theorems,  we used the conditions "$\alpha_j\equiv 1\pmod 4$" and "$\alpha_j\equiv -1\pmod 4$" (e.g. see the second item of Theorem \ref{first thm}). These conditions are in fact  well is characterized by the above proposition in terms of $p$ and $q$.
	\end{remark}

	\section{\bf The genus field of the real quartic cyclic fields $K=\mathbb{Q}(\sqrt{a  \varepsilon_p\sqrt{p}})$}\label{sec:4}
	
	Let $K=\mathbb{Q}(\sqrt{a\varepsilon_p\sqrt{p}})$ be a real cyclic quartic field such that: $p=2$ or $p\equiv 1\pmod 4$ is a prime number, $\varepsilon_p$ is the fundamental unit of $k_0=\mathbb{Q}(\sqrt{p})$ and $a$ is a square-free positive integer relatively prime to $p$. Put \begin{eqnarray}
		a=\displaystyle\prod_{i=1}^n q_i \text{  or } 2\displaystyle\prod_{i=1}^n q_i,
	\end{eqnarray} 
	where  $q_1, q_2, \dots, q_n$ are distinct odd primes.  
	Assume    that the Legendre symbols $\left(\dfrac{p}{q_j}\right)=1$ for $1\leq j \leq m$ $( m\leq n$, the case $m=0$ is included here$)$ and $\left(\dfrac{p}{q_j}\right)=-1$ for $m+1 \leq j \leq n$. If $q$ is divisible by a prime congruent to $3\pmod 4$, we put $$\displaystyle i_0=\min_{}\{i:q_i\equiv 3\pmod 4\}\label{io}.$$
	Let $e_\ell$ denote   the ramification index of a prime number $\ell$ in $K/\mathbb{Q}$. With these notations, we have:

	\subsection{The case  $p\equiv 1\pmod 8$:}

	\begin{proposition}\label{prop 4.1}
		Assume that $a$ is odd and   $p\equiv 1\pmod 8$. We have
		\begin{enumerate}[\rm 1.]
			\item If   $\forall  i\in \{1,...,n\}$, $q_i\equiv 1\pmod 4$, then 
			$$K^{(*)}=K( \sqrt{q_1},\sqrt{q_2},..., \sqrt{q_n}  ).$$
			
			\item If $a$ is divisible by a prime congruent to $3\pmod 4$ and $a \equiv 1\pmod 4$,   then 
			$$K^{(*)}=K(\{\sqrt{ q_i^*}\}_{i\not= i_0}  ).$$
			\item If $a$ is divisible by a prime congruent to $3\pmod 4$ and $a\equiv 3\pmod 4$,   then 
			$$K^{(*)}=K( \sqrt{q_1},\sqrt{q_2},..., \sqrt{q_n}  ). $$
		\end{enumerate}
		Such that for all $i\not= i_0 $, $q_i^*=q_i$ if $q_i\equiv 1\pmod 4$ and $q_i^*=q_{i_0}q_i$ else.
	\end{proposition}
	\begin{proof}
		We shall use the notations and  the preliminaries in the beginning of  the second section. We have
		\begin{enumerate}[\rm 1.]
			\item  	As $p\equiv 1\pmod 8$,  then $b\equiv 0\pmod 4$ (where $p=b^2+c^2$). Thus $a+b\equiv 1\pmod 4$ and $\delta_K=a^2p^3$ (cf. Lemmas \ref{lemma 2.1; discriminants} and \ref{lm 2.2}). Therefore,   the prime divisors of $\delta_K$   are the $n$ primes $q_i$ with ramification index equals $2$ and the prime $p$, with  $e_p=4$. Since $\mathbb{Q}(\sqrt{ \varepsilon_p\sqrt{p}})$ is the unique subfield of $\mathbb Q{(\zeta_p)}$ of degree $4$ over $\mathbb{Q}$, we have by Proposition \ref{prop 2.4}:

			If $a$ is as in the first item, we have 
			\begin{eqnarray*}
				K_{(*)}&=&\mathbb{Q}(\sqrt{{q_1}})\mathbb{Q}(\sqrt{{q_2}})\dots\mathbb{Q}(\sqrt{{q_n}})\mathbb{Q}(\sqrt{{ {\varepsilon}}_p\sqrt{p}})K\\
				&=& K(\sqrt{{q_1} }, \sqrt{{q_2} },..., \sqrt{{q_n} },\sqrt{{ {\varepsilon}}_p\sqrt{p}} )\\
				&=& K(\sqrt{q_1}, \sqrt{q_2 },..., \sqrt{q_n } ).	
			\end{eqnarray*}
			Thus $K^{(*)}={K_{(*)}}^+=K(\sqrt{q_1}, \sqrt{q_2 },..., \sqrt{q_n } )$.
			\item If $a$ is as in the second item, then as above,  the prime divisors of $\delta_K$   are the $n$ primes $q_i$ with ramification index equals $2$ and the prime $p$, with  $e_p=4$.  To simplify we reorder the primes $q_i$ to get    $\displaystyle\prod_{i=1}^n q_i=\displaystyle\prod_{i=1}^r q_i\displaystyle\prod_{i=r+1}^n q_i$, such that $\forall  i\in \{1,...,r\}$, $q_i\equiv 1\pmod 4$ and $\forall  i\in \{r+1,...,n\}$, $q_i\equiv 3\pmod 4$ (this is  in this proof only). We have
			\begin{eqnarray*}
				K_{(*)}&=&  K(\sqrt{q_1},..., \sqrt{q_r }, \sqrt{- q_{r+1} },..., \sqrt{-{q_n} },\sqrt{{ {\varepsilon}}_p\sqrt{p}} )\\
				&=& K(\sqrt{q_1},\sqrt{q_2 },..., \sqrt{q_r }, \sqrt{q_n q_{r+1} },...,\sqrt{q_n q_{n-1} }, \sqrt{{-q_n} } )	
			\end{eqnarray*}
			Thus $K^{{(*)}}={K_{(*)}}^+=K( \sqrt{q_1^*},\sqrt{q_2^*},..., \sqrt{q_{n-1}^*}  )$.
			\item In this case we have $a+b\equiv 3\pmod 4$. Thus,  $\delta_K=2^4a^2p^3$.  Therefore, the prime divisors of 
			$  \delta_K$ are: the $n$ primes $q_i$ with $e_{q_i}=2$, $p$ with $e_p=4$ and $2$ with $e_2=2$. As above we complete the proof.
		\end{enumerate}
	\end{proof}

	\begin{proposition}
		Assume that $a$ is even and   $p\equiv 1\pmod 8$. We have
		\begin{enumerate}[\rm 1.]
			\item If   $\forall  i\in \{1,...,n\}$, $q_i\equiv 1\pmod 4$, then 
			$$K^{(*)}=K(\sqrt{2}, \sqrt{q_1},\sqrt{q_2},..., \sqrt{q_n}  ).$$
			
			\item If $a$ is divisible by a prime congruent to $3\pmod 4$ and $ a \equiv 1\pmod 4$,    then   
			$$K^{(*)}=K(\sqrt{2}, \{\sqrt{q_i^*}\}_{i\not= i_0}  ).$$

			\item If $a$ is divisible by a prime congruent to $3\pmod 4$ and $ a \equiv 3\pmod 4$,   then   
			$$K^{(*)}=K( \{\sqrt{q_i^*}\}_{i\not= i_0}, \sqrt{{ {\varepsilon}}_p\sqrt{p}}  ).$$
		\end{enumerate}
		Such that for all $i\not= i_0 $, $q_i^*=q_i$ if $q_i\equiv 1\pmod 4$ and $q_i^*=q_{i_0}q_i$ else.
	\end{proposition}	
	\begin{proof}
		We proceed as in the proof of Proposition	\ref{prop 4.1}.
	\end{proof}

	\subsection{The case  $p\equiv 5\pmod 8$:}
	\begin{proposition}  	Assume that $a$ is odd and   $p\equiv 5\pmod 8$. We have
		\begin{enumerate}[\rm 1.]
			\item If $a=\displaystyle\prod_{i=1}^n q_i$ such that $\forall  i\in \{1,...,n\}$, $q_i\equiv 1\pmod 4$, then 
			$$K^{(*)}=K( \sqrt{q_1},\sqrt{q_2},..., \sqrt{q_n}  ).$$
			
			\item If $a$ is divisible by a prime congruent to $3\pmod 4$ and $a \equiv 1\pmod 4$,  then 
			$$K^{(*)}=K( \{\sqrt{q_i^*}\}_{i\not= i_0}, \sqrt{{ q_{i_0}{\varepsilon}}_p\sqrt{p}}  ).$$
			
			\item If $a$ is divisible by a prime congruent to $3\pmod 4$ and $a \equiv 3\pmod 4$, then 
			$$K^{(*)}=K( \{\sqrt{q_i^*}\}_{i\not= i_0}   ).$$
		\end{enumerate}
		Such that for all $i\in \{1,...,n-1\} $, $q_i^*=q_i$ if $q_i\equiv 1\pmod 4$ and $q_i^*=q_{i_0}q_i$ else.
	\end{proposition}
	\begin{proof}
		We proceed as in the proof of Proposition	\ref{prop 4.1}.
	\end{proof}

	\begin{proposition}
		Assume that $a$ is even and   $p\equiv 5\pmod 8$. We have
		\begin{enumerate}[\rm 1.]
			\item If   $\forall  i\in \{1,...,n\}$, $q_i\equiv 1\pmod 4$, then 
			$$K^{(*)}=K(  \sqrt{q_1},\sqrt{q_2},..., \sqrt{q_n}  ).$$
			
			\item If $a$ is divisible by a prime congruent to $3\pmod 4$ and $ a \equiv 1\pmod 4$,    then  
			$$K^{(*)}=K( \{\sqrt{q_i^*}\}_{i\not= i_0}, \sqrt{{ q_{i_0}{\varepsilon}}_p\sqrt{p}}  ).$$
			\item If $a$ is divisible by a prime congruent to $3\pmod 4$ and $ a \equiv 3\pmod 4$,   then   
			
			$$K^{(*)}=K( \{\sqrt{q_i^*}\}_{i\not= i_0}, \sqrt{{ q_{i_0}{\varepsilon}}_p\sqrt{p}}  ).$$
		\end{enumerate}
		Such that for all $i\not= i_0 $, $q_i^*=q_i$ if $q_i\equiv 1\pmod 4$ and $q_i^*=q_{i_0}q_i$ else.
	\end{proposition}	
	\begin{proof}
		We proceed as in the proof of Proposition	\ref{prop 4.1}.
	\end{proof}

	\subsection{The case  $p=2$:}
	\begin{proposition}  Assume that  $p=2$. We have                               
		\begin{enumerate}[\rm 1.]
			\item If $\forall  i\in \{1,...,n\}$, $q_i\equiv 1\pmod 4$, then 
			$$K^{(*)}=K( \sqrt{q_1},\sqrt{q_2},..., \sqrt{q_n}  ).$$
			\item If $a$ is divisible by a prime congruent to $3\pmod 4$ then	
			$$K^{(*)}=K( \{\sqrt{q_i^*}\}_{i\not= i_0}).$$
		\end{enumerate}
		Such that for all $i\not= i_0 $, $q_i^*=q_i$ if $q_i\equiv 1\pmod 4$ and $q_i^*=q_{i_0}q_i$ else.
	\end{proposition}	
	\begin{proof}
		We proceed as in the proof of Proposition	\ref{prop 4.1}.
	\end{proof}
	
	\section{\bf The Hilbert genus field of the real cyclic quartic  fields $K=\mathbb{Q}(\sqrt{a  \varepsilon_p\sqrt{p}})$}\label{sec:5}
	
	Keep the notations in the beginning of the above section. If $p\equiv 1\pmod4$  and $1\leq j\leq m$, we set $$\alpha_j=\left\{ \begin{array}{lcl}
		x_j+y_j\sqrt{p}, &\text{ if }& q_j\equiv 1\pmod 4,\\
		(x_j+y_j\sqrt{p})\sqrt{p}, &\text{ if }& q_j\equiv 3\pmod 4,
	\end{array}\right.$$
	where $x_j$ and $ y_j$ are natural integers satisfying the Diophantine equation:	
	\begin{eqnarray}
		x^2-py^2=q_j^{\lambda h}, 
	\end{eqnarray}
	where $h$ is the class number of  $\mathbb{Q}(\sqrt{p})$ and  $\lambda=\begin{cases}1,& \text{ if } p\equiv 1\pmod 8,\\
		3, & \text{ if }  p\equiv 5\pmod 8.\end{cases}$\\
	
	Note that   $\alpha_j\equiv x_j+y_j\equiv \pm 1 \pmod 4$, for $j=1,...,m$.
	If   there is some $j\in\{1,...,m\}$ such that 
	$\alpha_j\equiv -1\pmod 4$, we set $$j_0=\min \{j: \alpha_j\equiv -1\pmod 4\}.\label{j0}$$

	Now we can start the construction of the Hilbert genus field of the real cyclic quartic  fields $K=\mathbb{Q}(\sqrt{a  \varepsilon_p\sqrt{p}})$.	
	\subsection{\bf The case  $p\equiv 1\pmod 8$:}	
	\begin{theorem}\label{first thm}Suppose that $a$ is odd and $p\equiv 1\pmod 8$.
		\begin{enumerate}[\rm 1.]
			\item Assume that $\forall i\in \{1,...,n\}$  $q_i\equiv 1\pmod 4$, we have:
			\begin{enumerate}[\rm $\bullet$]
				\item If $m=0$, then $$E(K)=K(\sqrt{q_1},..., \sqrt{q_n}) .$$
				\item If $m\geq 1$ and  $\forall j \in\{1,...,m\}$  
				$\left(\dfrac{p}{q_j}\right)_4=\left(\dfrac{q_j}{p}\right)_4$, then 
				$$E(K)=K(\sqrt{q_1},..., \sqrt{q_n}, \sqrt{\alpha_1},...,\sqrt{\alpha_m}) .$$
				
				\item If $m\geq 1$ and $\exists j\{1,...,m\}$, $\left(\dfrac{p}{q_j}\right)_4\not=\left(\dfrac{q_j}{p}\right)_4$, then 
				
				$$E(K)=K(\sqrt{q_1},..., \sqrt{q_n},   \{\sqrt{\alpha_j^*} \}_{j\not= j_0}  ). $$
			\end{enumerate}

			\item Assume that $\exists i\in \{1,...,n\}$  $q_i\equiv 3\pmod 4$ and $a\equiv 1\pmod 4$, we have:
			\begin{enumerate}[\rm $\bullet$]
				\item If $m=0$, then 	$$E(K)=K(\{\sqrt{ q_i^*} \}_{i\not= i_0}  ) $$ 
				\item  If $m\geq 1$ and $\forall j  \in\{1,...,m\}$ $q_j\equiv 1\pmod 4$, we have 
				\begin{enumerate}[\rm $\star$]
					\item  If   $\forall j\in \{1,...,m\}$  
					$\alpha_j \equiv 1\pmod 4$, then  
					$$E(K)=K(\{\sqrt{ q_i^*} \}_{i\not= i_0} ,    \sqrt{\alpha_1},..., \sqrt{\alpha_m}   ). $$ 
					\item   If   $\exists j\in \{1,...,m\}$  
					$\alpha_j \equiv -1\pmod 4$, then  
					$$E(K)=K(\{\sqrt{ q_i^*} \}_{i\not= i_0} ,     \{\sqrt{ \alpha_j^*} \}_{j\not= j_0}, \sqrt{\varepsilon_p}). $$ 
				\end{enumerate}
				\item   If $m\geq 1$ and $\exists j  \in\{1,...,m\}$ $q_j\equiv 3\pmod 4$, we have 
				\begin{enumerate}[\rm $\star$]
					\item  If   $\forall j\in \{1,...,m\}$  
					$\alpha_j \equiv 1\pmod 4$, then  
					
					$$E(K)= 
					K(\{\sqrt{ q_i^*} \}_{i\not= i_0} ,   \sqrt{\alpha_1},...,\sqrt{\alpha_{m-1}}  )  .$$
					\item  If   $\exists j\in \{1,...,m\}$  
					$\alpha_j \equiv -1\pmod 4$, then 
					$$E(K)= 
					K(\{\sqrt{ q_i^*} \}_{i\not= i_0} ,    \{\sqrt{ \alpha_j^*} \}_{j\not= j_0}) . $$
				\end{enumerate}
			\end{enumerate}

			\item Assume that $\exists i\in \{1,...,n\}$  $q_i\equiv 3\pmod 4$ and $a\equiv 3\pmod 4$, we have: 
			\begin{enumerate}[\rm $\bullet$]
				\item  If $m=0$, then	 $$E(K)=K(\sqrt{q_1},..., \sqrt{q_n}) .$$
				\item    If $m\geq 1$ and $\forall j\in \{1,...,m\}$  
				$\alpha_j \equiv 1\pmod 4$, then  
				$$E(K)=K(\sqrt{q_1},..., \sqrt{q_n}, \sqrt{\alpha_1},...,\sqrt{\alpha_m}). $$
				\item   If $m\geq 1$ and  $\exists j\in \{1,...,m\}$  
				$\alpha_j \equiv -1\pmod 4$, then  
				$$E(K)=K(\sqrt{q_1},..., \sqrt{q_n} , \{\sqrt{ \alpha_j^*} \}_{j\not= j_0}, \sqrt{\varepsilon_p}) .$$ 
			\end{enumerate}
			
		\end{enumerate}
		where $\alpha_j^*=\alpha_j$ if $\alpha_j \equiv  1 \pmod 4$  and $\alpha_j^*=\alpha_{j_0}\alpha_j$ else.  
	\end{theorem}	
	\begin{proof}
		We shall prove the theorem by following the order of  the  cases and  the results given therein. 	
		\begin{enumerate}[\rm 1.]
			\item 
			
			\begin{enumerate}[\rm $\bullet$]
				\item Note that by Lemma \ref{lemma 4.1}  and  \cite[Theorem 4.6]{aziziTmmZekhII}, we have $r_2(\Delta(K)/{K^*}^2)=n$.
				Since 	$\beta= \displaystyle\prod_{i=1}^n {q_i^*}^{a_i} $ is not a square in $K$ for all $a_i\in\{0,1\}$ not all zero,
				the set	$\{q_1,...,q_n\}$ is linearly independent modulo $(K^*)^2$ (provided that the notion
				of linear independence is translated to a multiplicative setting: $\gamma_1,..., \gamma_s$ are multiplicatively independent if $\gamma_1^{m_1}... \gamma_s^{m_s}= 1$ implies that
				$m_i=0$, for all $i$).
				By	Proposition \ref{prop 4.1}, the extension $K(\sqrt{q_i})/K$ is unramified for all $i=1,...,n$. Therefore, $\{q_1,...,q_n\}$ is a generator set of $\Delta(K)$. So the result by the definition of $E(K)$ (cf. Section \ref{sec:1}).
				
				\item  By Lemma \ref{lemma 4.1} and  \cite[Theorem 4.6]{aziziTmmZekhII}, we have $r_2(\Delta(K)/{K^*}^2)=n+m$. Consider the following set
				$$\eta =\{q_1,...,q_n, \alpha_1,...,\alpha_m \}.$$
				Let us show that $\eta$ is linearly independent modulo $(K^*)^2$.\\
				Let   $\beta =\left(\displaystyle\prod_{i=1}^n {q_i^*}^{a_i}\right)\left(\displaystyle\prod_{j=1}^m {\alpha_j^*}^{b_j}\right)$, where $a_i, b_j \in \{0,1\}$ and are not all zero. Assume that $\beta \in {K^*}^2 .$ Thus, $ \beta\in k_0^2$ or $\delta  \beta\in k_0^2$, where 
				$k_0=\mathbb{Q}(\sqrt{p})$ and $\delta=a\varepsilon_p\sqrt{p}$. Note that there is some $j$ such that $b_j\not= 0$ (in fact, if else, we get $\left(\displaystyle\prod_{i=1}^n {q_i^*}^{a_i}\right)$ is a square in $K$, which is not true).

				If $  \beta\in k_0^2$, then $N(\beta):=N_{k_0/\mathbb{Q}}(\beta) \in \mathbb{Q}^2$. Since 
				$N(\alpha_j)=q_jz_j^2$, with $z_j=q_j^{ \frac{h-1}{2}}$ ($\frac{h-1}{2}$ is even).
				So one can verify that $N(\beta)=\left(\prod_{j=1}^m {q_j^*}^{b_j}\right)Z^2\in \mathbb{Q}^2$, for some $Z \in \mathbb{Q}$. This implies that $\left(\prod_{j=1}^m {q_j^*}^{b_j}\right)$ is a square in
				$\mathbb{Q}$, which is impossible.

				If $  \delta  \beta\in k_0^2$, as above one gets $-p\left(\prod_{j=1}^m {q_j^*}^{b_j}\right)$ is a square in
				$\mathbb{Q}$. Which is also impossible. Therefore, the elements of $\eta$ are 
				linearly independents modulo $(K^*)^2$.

				On the other hand, by Proposition \ref{prop 4.1}, $K(\sqrt{q_i})/K$ is an unramified extension for all $i$. 
				
				Since $\forall j \in\{1,...,m\}$  
				$\left(\dfrac{p}{q_j}\right)_4=\left(\dfrac{q_j}{p}\right)_4$, $\alpha_j\equiv 1\pmod 4$. As the primes $q_j$ ramify in $K/k_0$, then the ideal $(\alpha_j)$,  generated by $\alpha_j$ is a square of an ideal of $K$. So by Proposition \ref{unramifed quad ex}, $K(\sqrt{\alpha_j}) /K$ is unramified for all $j=1,...,m$. Therefore, 
				$\eta$ is a generator set of the group quotient $\Delta(K)/ {K^*}^2$. Hence 
				$$E(K)=K(\sqrt{q_1},..., \sqrt{q_n}, \sqrt{\alpha_1},...,\sqrt{\alpha_m}) .$$   
				
				\item  If $\exists j\in \{1,...,m\}$, $\left(\dfrac{p}{q_j}\right)_4\not=\left(\dfrac{q_j}{p}\right)_4$, then by Lemma \ref{lemma 4.1} and  \cite[Theorem 4.6]{aziziTmmZekhII}, we have $r_2(\Delta(K)/{K^*}^2)=n+m-1$. Consider the following set
				$$\eta =\{\sqrt{q_1},..., \sqrt{q_n}\}\cup   \{\sqrt{\alpha_j^*} \}_{j\not= j_0}  .$$
				Noting that $N(\alpha_j^*)=  \left\{ \begin{array}{l}
					q_jz_j^2  \text{ if } \alpha_j\equiv 1\pmod 4\\
					q_jq_{j_0}( z_{j_0}z_j)^2{p}^2 \text{ if } \alpha_j\equiv 3\pmod 4;
				\end{array}\right.$
				we show, as above,  that $\eta$ is linearly independent modulo   $(K^*)^2$ and therefore
				$$E(K)=K(\sqrt{q_1},..., \sqrt{q_n},   \{\sqrt{\alpha_j^*} \}_{j\not= j_0}  ) .$$
			\end{enumerate}

			\item  
			\begin{enumerate}[\rm $\bullet$]
				\item By  Lemma \ref{lemma 4.1} and \cite[Theorems 4.9 and 4.11]{aziziTmmZekhII}, we have $r_2(\Delta(K)/{K^*}^2)=n-1$. Therefore, as in the first item we get the result of the point.
				\item   Assume that $\exists i\in \{1,...,n\}$  $q_i\equiv 3\pmod 4$ and $a\equiv 1\pmod 4$, we have:
				\begin{enumerate}[\rm $\star$]
					\item   As above we show that the set  $  \{\sqrt{ q_i^*} \}_{i\not= i_0}\cup \{\sqrt{\alpha_1},..., \sqrt{\alpha_m}   \}$ is a generator set of $\Delta(K)/{K^*}^2$. Therefore, $$E(K)=K(\{\sqrt{ q_i^*} \}_{i\not= i_0} ,    \sqrt{\alpha_1},..., \sqrt{\alpha_m}   ). $$
					\item  In this case,  Lemma \ref{lemma 4.1} and  \cite[Theorem 4.11]{aziziTmmZekhII}, we have $r_2(\Delta(K)/{K^*}^2)=n+m-1$. We shall consider the following set :
					$$\eta = \{q_i^* \}_{i\not= i_0}\cup   \{\alpha_j^* \}_{j\not= j_0}\cup \{\varepsilon_p  \}.$$
					Let us show that $\eta$ is linearly independent modulo   $(K^*)^2$.\\
					Let   $\beta =\left(\displaystyle\prod_{i=1}^n {q_i^*}^{a_i}\right)\left(\displaystyle\prod_{j=1}^m {\alpha_j^*}^{b_j}\right)\varepsilon_p^c$, where $a_i, b_j, c \in \{0,1\}$ and are not all zero. Assume that $\beta \in {K^*}^2 .$ We have,  $N(\beta)=\left(\pm (p)^{\ell_1}(q_1)^{\ell_2}\prod_{j=1}^m {\alpha_j^*}^{b_j}\right)\cdot Z^2 \in \mathbb{Q}^2$, which is impossible. We similarly show that $N(\delta\beta)$ can not be in $\mathbb{Q}^2$. 
					Therefore, $\eta$ is linearly independent modulo   $(K^*)^2$. 
					
					As above we easily check that $K(\sqrt{q_i^*})/K$ $(1\leq i \leq n)$ and $K(\sqrt{\alpha_j^*})/K$ $(1\leq j \leq m)$ are unramified extensions.  By \cite[p. 67]{cohn}  that the extension  $K(\sqrt{\varepsilon_p})/K$ is    unramified. Therefore, the set $\eta$ is linearly independent modulo   $(K^*)^2$. Hence 
					$$E(K)=K(\{\sqrt{ q_i^*} \}_{i\not= i_0} ,     \{\sqrt{ \alpha_j^*} \}_{j\not= j_0}, \sqrt{\varepsilon_p}). $$ 
					
				\end{enumerate}
				\item  In this case, by Lemma \ref{lemma 4.1} and  \cite[Theorems 4.9 and 4.11]{aziziTmmZekhII}
				$r_2(\Delta(K)/{K^*}^2)=n+m-2$.  
				\begin{enumerate}[\rm $\star$]
					\item   If   $\forall j\in \{1,...,m\}$  
					$\alpha_j\equiv1\pmod 4$,  then, as above we show that 
					$\eta=\{\sqrt{ q_i^*} \}_{i\not= i_0}\cup \{\sqrt{\alpha_1},...,\sqrt{\alpha_{m-1}}\}$ is a generator set of $\Delta(K)/{K^*}^2$.
					\item If   $\exists j\in \{1,...,m\}$  
					$\alpha_j\equiv1-\pmod 4$, then, as above we show that 
					$\eta=\{\sqrt{ q_i^*} \}_{i\not= i_0} \cup    \{\sqrt{ \alpha_j^*} \}_{j\not= j_0}$ is a generator set of $\Delta(K)/{K^*}^2$.
				\end{enumerate}

			\end{enumerate}

			\item 
			\begin{enumerate}[\rm $\bullet$]
				\item $m=0$, the result is direct.
				\item     If 	$m\geq 1$  and $\forall j\in \{1,...,m\}$  
				$\alpha\equiv 1 \pmod 4$, 
				then, as above we show that $$\eta=\{\sqrt{ q_i^*} \}_{i\not= i_0} \cup \{     \sqrt{\alpha_1},..., \sqrt{\alpha_m},\sqrt{\varepsilon_p}  \}$$ is a generator set of $\Delta(K)/{K^*}^2$ (use Lemma \ref{lemma 4.1} and \cite[Theorems 4.8 and 4.10]{aziziTmmZekhII}).  
				
				\item   In the last case,   we show that $\eta= \{\sqrt{ q_i^*} \}_{i\not= i_0}\cup \{ \sqrt{\alpha_1},..., \sqrt{\alpha_m},\sqrt{\varepsilon_p}  \}$ is a generator set of $\Delta(K)/{K^*}^2$ (use Lemma \ref{lemma 4.1} and \cite[Theorems 4.8 and 4.10]{aziziTmmZekhII}).  
				
			\end{enumerate}
			
		\end{enumerate}	
		Which completes the proof.	
	\end{proof}

	\begin{theorem} 		Suppose that $a$ is even and $p\equiv 1\pmod 8$.
		\begin{enumerate}[\rm 1.]
			\item Assume that $\forall \in \{1,...,n\}$  $q_i\equiv 1\pmod 4$, we have:
			\begin{enumerate}[\rm $\bullet$]
				\item If $m=0$ and   $\left(\dfrac{2}{p}\right)_4=(-1)^{\frac{p-1}{8}}$, then 
				$$E(K)=K(\sqrt{2},\sqrt{q_1},..., \sqrt{q_n},\sqrt{\varepsilon_p}). $$
				
				\item If $m=0$ and  $\left(\dfrac{2}{p}\right)_4\not=(-1)^{\frac{p-1}{8}}$, then 
				$$E(K)=K(\sqrt{2},\sqrt{q_1},..., \sqrt{q_n} ). $$

				\item If $m\geq1$, $\forall j \in\{1,...,m\}$  
				$\left(\dfrac{p}{q_j}\right)_4=\left(\dfrac{q_j}{p}\right)_4$ and $\left(\dfrac{2}{p}\right)_4=\left(\dfrac{p}{2}\right)_4$, then 
				$$E(K)=K(\sqrt{2},\sqrt{q_1},..., \sqrt{q_n}, \sqrt{\alpha_1},...,\sqrt{\alpha_m},\sqrt{\varepsilon_p}). $$
				
				\item If $m\geq1$, $\forall j\{1,...,m\}$, $\left(\dfrac{p}{q_j}\right)_4=\left(\dfrac{q_j}{p}\right)_4$ and $\left(\dfrac{2}{p}\right)_4\not=\left(\dfrac{p}{2}\right)_4$, then 
				$$E(K)=K(\sqrt{2},\sqrt{q_1},..., \sqrt{q_n}, \sqrt{\alpha_1},...,\sqrt{\alpha_m} ). $$
				\item If $m\geq1$, $\exists j\{1,...,m\}$, $\left(\dfrac{p}{q_j}\right)_4\not=\left(\dfrac{q_j}{p}\right)_4$  then 
				$$E(K)=K(\sqrt{2},\sqrt{q_1},..., \sqrt{q_n},   \{\sqrt{\alpha_j^*} \}_{j\not= j_0},\sqrt{\varepsilon_p}  ). $$
				
			\end{enumerate}

			\item Assume that $\exists i\in \{1,...,n\}$  $q_i\equiv 3\pmod 4$ and $a\equiv 1\pmod 4$, we have:
			\begin{enumerate}[\rm $\bullet$]
				\item If $m=0$, then 
				$$E(K)= 
				K(\sqrt{2}, \{\sqrt{ q_i^*} \}_{i\not= i_0},\sqrt{ \varepsilon_p}   ).  $$

				\item  If $m\geq1$ and  $\forall j  \in\{1,...,m\}$ $q_j\equiv 1\pmod 4$, we have 
				\begin{enumerate}[\rm $\star$]
					\item  If   $\forall j\in \{1,...,m\}$  
					$\alpha_j\equiv1\pmod 4$, then  
					$$E(K)=K(\sqrt{2},\{\sqrt{ q_i^*} \}_{i\not= i_0} ,    \sqrt{\alpha_1},..., \sqrt{\alpha_m}, \sqrt{\varepsilon_p}  ). $$ 
					\item   If   $\exists j\in \{1,...,m\}$  
					$\alpha_j\equiv-1\pmod 4$, then  
					$$E(K)=K(\sqrt{2},\{\sqrt{ q_i^*} \}_{i\not= i_0} ,   \sqrt{\beta_1},..., \sqrt{\beta_m},, \sqrt{\varepsilon_p}), $$ 
					where $\beta_j=\alpha_j$ if $\alpha_j\equiv 1\pmod 4$  and $\beta_j=q_{i_0}\alpha_j$ elsewhere.
				\end{enumerate}
				\item   If  $m\geq1$ and $\exists j  \in\{1,...,m\}$ $q_j\equiv 3\pmod 4$, we have 
				\begin{enumerate}[\rm $\star$]
					\item   If   $\forall j\in \{1,...,m\}$  
					$\alpha_j\equiv1\pmod 4$, then  
					
					$$E(K)= 
					K(\sqrt{2},\{\sqrt{ q_i^*} \}_{i\not= i_0} ,   \sqrt{\alpha_1},...,\sqrt{\alpha_{m-1}}  ) . $$
					\item   If   $\exists j\in \{1,...,m\}$  
					$\alpha_j\equiv-1\pmod 4$, then 
					$$E(K)= 
					K(\sqrt{2},\{\sqrt{ q_i^*} \}_{i\not= i_0} ,    \{\sqrt{ \alpha_j^*} \}_{j\not= j_0}, \sqrt{\varepsilon_p}) . $$
				\end{enumerate}
			\end{enumerate}

			\item Assume that $\exists i\in\{1,...,n\}$  $q_i\equiv 3\pmod 4$ and $a\equiv 3\pmod 4$, we have: 
			\begin{enumerate}[\rm $\bullet$]
				\item  If $m=0$, then
				$$E(K)=K(\{\sqrt{ q_i^*} \}_{i\not= i_0} ,   \sqrt{ {\varepsilon_p}\sqrt{p}}  ). $$
				\item If  $m\geq1$ and  $\forall j\in \{1,...,m\}$  
				$\alpha_j\equiv1\pmod 4$, then  
				$$E(K)=K(\{\sqrt{ q_i^*} \}_{i\not= i_0} ,    \sqrt{\alpha_1},..., \sqrt{\alpha_m},\sqrt{\varepsilon_p\sqrt{p}}  ). $$ 
				\item If  $m\geq1$ and   $\exists j\in \{1,...,m\}$  
				$\alpha_j\equiv-1\pmod 4$, then  
				$$E(K)=K(\{\sqrt{ q_i^*} \}_{i\not= i_0} ,     \{\sqrt{ \alpha_j^*} \}_{j\not= j_0}, \sqrt{\varepsilon_p},\sqrt{\varepsilon_p\sqrt{p}} ). $$ 
			\end{enumerate}
			
		\end{enumerate}
		where $\alpha_j^*=\alpha_j$ if $\alpha_j \equiv  1 \pmod 4$  and $\alpha_j^*=\alpha_{j_0}\alpha_j$ else.
	\end{theorem}	
	
	\begin{proof}
		The detailed proof of this theorem is long, however, inspired by the proof of Theorem \ref{first thm} and  using the same references therein,  a patient reader can  prove this theorem. 
	\end{proof}

	\subsection{The case  $p\equiv 5\pmod 8$:}

	\begin{theorem}Suppose that $a$ is odd and $p\equiv 5\pmod 8$.
		\begin{enumerate}[\rm 1.]
			
			\item Assume that $\forall i\in \{1,...,n\}$  $q_i\equiv 1\pmod 4$, we have:
			\begin{enumerate}[\rm $\bullet$]
				\item If $m=0$, then
				$$E(K)=K(\sqrt{q_1},..., \sqrt{q_n} ). $$
				\item If  $m\geq1$ and $\forall j \in\{1,...,m\}$, $\left(\dfrac{p}{q_j}\right)_4=\left(\dfrac{q_j}{p}\right)_4$, then 
				$$E(K)=K(\sqrt{q_1},..., \sqrt{q_n}, \sqrt{\alpha_1},...,\sqrt{\alpha_m}). $$
				
				\item If  $m\geq1$ and $\exists j\{1,...,m\}$, $\left(\dfrac{p}{q_j}\right)_4\not=\left(\dfrac{q_j}{p}\right)_4$, then 
				$$E(K)=K(\sqrt{q_1},..., \sqrt{q_n},   \{\sqrt{\alpha_j^*} \}_{j\not= j_0},\sqrt{\varepsilon_p}  ). $$
			\end{enumerate}

			\item Assume that $\exists i\in \{1,...,n\}$  $q_i\equiv 3\pmod 4$ and $a\equiv 1\pmod 4$, we have:		
			\begin{enumerate}[\rm $\bullet$]
				\item If $m=0$, then
				$$E(K)=K( \{\sqrt{q_i^*}\}_{i\not= i_0}, \sqrt{{ q_{i_0}{\varepsilon}}_p\sqrt{p}}  ).$$
				\item  If  $m\geq1$ and $\forall j  \in\{1,...,m\}$ $q_j\equiv 1\pmod 4$, we have 
				\begin{enumerate}[\rm $\star$]
					\item  If   $\forall j\in \{1,...,m\}$  
					$\alpha_j\equiv1\pmod 4$, then  
					$$E(K)=K(\{\sqrt{ q_i^*} \}_{i\not= i_0} ,    \sqrt{\alpha_1},..., \sqrt{\alpha_m}, \sqrt{q_{i_0}\varepsilon_p\sqrt{p}}   ). $$ 
					\item   If   $\exists j\in \{1,...,m\}$  
					$\alpha_j\equiv-1\pmod 4$, then  
					$$E(K)=K(\{\sqrt{ q_i^*} \}_{i\not= i_0} ,     \{\sqrt{ \alpha_j^*} \}_{j\not= j_0}, \sqrt{\varepsilon_p},\sqrt{q_{i_0}\varepsilon_p\sqrt{p}}). $$ 
				\end{enumerate}
				\item  If  $m\geq1$ and $\exists j  \in\{1,...,m\}$ $q_j\equiv 3\pmod 4$, we have 
				\begin{enumerate}[\rm $\star$]
					\item  If   $\forall j\in \{1,...,m\}$  
					$\alpha_j\equiv1\pmod 4$, then  
					
					$$E(K)= 
					K(\{\sqrt{ q_i^*} \}_{i\not= i_0} ,   \sqrt{\alpha_1},...,\sqrt{\alpha_{m-1}},\sqrt{q_{i_0}\varepsilon_p\sqrt{p}}  ).  $$
					\item  If   $\exists j\in \{1,...,m\}$  
					$\alpha_j\equiv-1\pmod 4$, then 
					$$E(K)= 
					K(\{\sqrt{ q_i^*} \}_{i\not= i_0} ,    \{\sqrt{ \alpha_j^*} \}_{j\not= j_0},\sqrt{q_{i_0}\varepsilon_p\sqrt{p}}).  $$
				\end{enumerate}
			\end{enumerate}

			\item Assume that $\exists i\in \{1,...,n\}$  $q_i\equiv 3\pmod 4$ and $a\equiv 3\pmod 4$, we have:
			\begin{enumerate}[\rm $\bullet$]
				\item If $m=0$, then
				$$E(K)=K( \{\sqrt{q_i^*}\}_{i\not= i_0})$$
				\item If  $m\geq1$ and  $\forall j\in \{1,...,m\}$, $q_j\equiv 1\pmod 4$, we have
				\begin{enumerate}[\rm $\star$]
					\item    If  $\forall j\in \{1,...,m\}$  
					$\alpha_j\equiv1\pmod 4$, then  
					$$E(K)=K(\{\sqrt{ q_i^*} \}_{i\not= i_0} ,    \sqrt{\alpha_1},..., \sqrt{\alpha_{m}}  ). $$ 
					\item   If   $\exists j\in \{1,...,m\}$  
					$\alpha_j\equiv-1\pmod 4$, then  
					$$E(K)=K(\{\sqrt{ q_i^*} \}_{i\not= i_0} ,     \{\sqrt{ \alpha_j^*} \}_{j\not= j_0},\sqrt{\varepsilon_p}). $$ 
				\end{enumerate}
				\item If  $m\geq1$ and   $\exists j\in \{1,...,m\}$, $q_j\equiv 3\pmod 4$, we have
				\begin{enumerate}[\rm $\star$]
					\item    If  $\forall j\in \{1,...,m\}$  
					$\alpha_j\equiv1\pmod 4$, then  
					$$E(K)=K(\{\sqrt{ q_i^*} \}_{i\not= i_0} ,    \sqrt{\alpha_1},..., \sqrt{\alpha_{m-1}}  ). $$ 
					\item   If   $\exists j\in \{1,...,m\}$  
					$\alpha_j\equiv-1\pmod 4$, then  
					$$E(K)=K(\{\sqrt{ q_i^*} \}_{i\not= i_0} ,     \{\sqrt{ \alpha_j^*} \}_{j\not= j_0}). $$ 
				\end{enumerate}
			\end{enumerate}	
		\end{enumerate}
		where $\alpha_j^*=\alpha_j$ if $\alpha_j \equiv  1 \pmod 4$   and $\alpha_j^*=\alpha_{j_0}\alpha_j$.
	\end{theorem}	
	\begin{proof}
		Inspired by the proof of Theorem \ref{first thm},   the  patient reader can   prove this theorem by using Lemma \ref{lemma 4.1} and \cite{aziziTmmZekhI}. 
	\end{proof}

	\begin{theorem} 	Suppose that $a$ is even and $p\equiv 5\pmod 8$.
		\begin{enumerate}[\rm 1.] 
			\item Assume that $\forall i \in \{1,...,n\}$  $q_i\equiv 1\pmod 4$, we have:
			\begin{enumerate}[\rm $\bullet$]
				\item If $m=0$, then
				$$E(K)=K(\sqrt{q_1},..., \sqrt{q_n})$$
				\item If  $m\geq1$ and $\forall j \in\{1,...,m\}$,  
				$\left(\dfrac{p}{q_j}\right)_4=\left(\dfrac{q_j}{p}\right)_4$, then 
				$$E(K)=K(\sqrt{q_1},..., \sqrt{q_n}, \sqrt{\alpha_1},...,\sqrt{\alpha_m}). $$

				\item If  $m\geq1$ and $\exists j\{1,...,m\}$, $\left(\dfrac{p}{q_j}\right)_4\not=\left(\dfrac{q_j}{p}\right)_4$, then 
				$$E(K)=K(\sqrt{q_1},..., \sqrt{q_n},   \{\sqrt{\alpha_j^*} \}_{j\not= j_0},\sqrt{\varepsilon_p}  ) $$
				
			\end{enumerate}

			\item Assume that $\exists i\in \{1,...,n\}$  $q_i\equiv 3\pmod 4$ and $a\equiv 1\pmod 4$, we have:
			\begin{enumerate}[\rm $\bullet$]
				\item If $m=0$, then
				$$E(K)=K( \{\sqrt{q_i^*}\}_{i\not= i_0}, \sqrt{{ q_{i_0}{\varepsilon}}_p\sqrt{p}}  ).$$
				\item  If  $m\geq1$ and $\forall j  \in\{1,...,m\}$ $q_j\equiv 1\pmod 4$, we have 
				\begin{enumerate}[\rm $\star$]
					\item  If   $\forall j\in \{1,...,m\}$  
					$\alpha_j\equiv1\pmod 4$, then  
					$$E(K)=K(\{\sqrt{ q_i^*} \}_{i\not= i_0} ,    \sqrt{\alpha_1},..., \sqrt{\alpha_m},\sqrt{q_{i_0}\varepsilon_p\sqrt{p}}   ) $$ 
					\item   If  $\exists j\in \{1,...,m\}$  
					$\alpha_j\equiv-1\pmod 4$, then  
					$$E(K)=K(\{\sqrt{ q_i^*} \}_{i\not= i_0} ,     \{\sqrt{ \alpha_j^*} \}_{j\not= j_0}, \sqrt{\varepsilon_p},\sqrt{q_{i_0}\varepsilon_p\sqrt{p}}) $$ 
				\end{enumerate}
				\item   If  $m\geq1$ and $\exists j  \in\{1,...,m\}$ $q_j\equiv 3\pmod 4$, we have 
				\begin{enumerate}[\rm $\star$]
					\item  If   $\forall j\in \{1,...,m\}$  
					$\alpha_j\equiv1\pmod 4$, then  
					
					$$E(K)= 
					K(\{\sqrt{ q_i^*} \}_{i\not= i_0} ,   \sqrt{\alpha_1},...,\sqrt{\alpha_{m-1}},\sqrt{q_{i_0}\varepsilon_p\sqrt{p}}  )  $$
					\item  If   $\exists j\in \{1,...,m\}$  
					$\alpha_j\equiv-1\pmod 4$, then 
					$$E(K)= 
					K(\{\sqrt{ q_i^*} \}_{i\not= i_0} ,    \{\sqrt{ \alpha_j^*} \}_{j\not= j_0},\sqrt{q_{i_0}\varepsilon_p\sqrt{p}})  $$
				\end{enumerate}
			\end{enumerate}

			\item Assume that $\exists i\in\{1,...,n\}$  $q_i\equiv 3\pmod 4$ and $a\equiv 3\pmod 4$, we have:
			\begin{enumerate}[\rm $\bullet$] 
				\item If $m=0$, then
				$$E(K)=K( \{\sqrt{q_i^*}\}_{i\not= i_0}, \sqrt{{ q_{i_0}{\varepsilon}}_p\sqrt{p}}  ).$$
				\item If  $m\geq1$ and $\forall j  \in\{1,...,m\}$ $q_j\equiv 1\pmod 4$, we have 	
				\begin{enumerate}[\rm $\star$]
					\item  If $\forall j\in \{1,...,m\}$  
					$\alpha_j\equiv1\pmod 4$, then  
					$$E(K)=K(\{\sqrt{ q_i^*} \}_{i\not= i_0} ,    \sqrt{\alpha_1},..., \sqrt{\alpha_m}, \sqrt{{ q_{i_0}{\varepsilon}}_p\sqrt{p}}  ). $$ 
					\item    If   $\exists j\in \{1,...,m\}$  
					$\alpha_j\equiv-1\pmod 4$, then  
					$$E(K)=K(\{\sqrt{ q_i^*} \}_{i\not= i_0} ,     \{\sqrt{ \alpha_j^*} \}_{j\not= j_0}, \sqrt{\varepsilon_p}, \sqrt{{ q_{i_0}{\varepsilon}}_p\sqrt{p}} ) $$ 
				\end{enumerate}
				\item If  $m\geq1$ and $\exists j  \in\{1,...,m\}$ $q_j\equiv 3\pmod 4$, we have 
				\begin{enumerate}[\rm $\star$]
					\item    If  $\forall j\in \{1,...,m\}$  
					$\alpha_j\equiv1\pmod 4$, then  
					$$E(K)=K(\{\sqrt{ q_i^*} \}_{i\not= i_0} ,    \sqrt{\alpha_1},..., \sqrt{\alpha_{m-1}}, \sqrt{{ q_{i_0}{\varepsilon}}_p\sqrt{p}}  ) $$ 
					\item   If   $\exists j\in \{1,...,m\}$  
					$\alpha_j\equiv-1\pmod 4$, then  
					$$E(K)=K(\{\sqrt{ q_i^*} \}_{i\not= i_0} ,     \{\sqrt{ \alpha_j^*} \}_{j\not= j_0}, \sqrt{{ q_{i_0}{\varepsilon}}_p\sqrt{p}} ) $$ 
				\end{enumerate}
			\end{enumerate}
		\end{enumerate}
		where $\alpha_j^*=\alpha_j$ if $\alpha_j \equiv  1 \pmod 4$  and $\alpha_j^*=\alpha_{j_0}\alpha_j$. 
	\end{theorem}	
	\begin{proof}
		Inspired by the proof of Theorem \ref{first thm},   the  patient reader can prove this theorem by using Lemma \ref{lemma 4.1} and \cite{aziziTmmZekhI}. 
	\end{proof}
	
	\subsection{The case  $p=2$:}

	\begin{theorem} Suppose that $a$ and $p=2$. For   $j$ such that   $1\leq j \leq m$, let   $x_j$ and $y_j$ be the  two   positive integers such that $q_j=x_j^2-2y_j^2$.  Put $\alpha_j = x_j+y_j\sqrt{2} $, then  
		\begin{enumerate}[\rm 1.]
			\item Assume that $\forall i \in \{1,...,n\}$  $q_i\equiv 1\pmod 4$, we have:
			\begin{enumerate}[\rm $\bullet$]
				\item If $m=0$, then
				$$E(K)=K(\sqrt{q_1},..., \sqrt{q_n}).$$  
				\item If  $m\geq1$ and $\forall j \in\{1,...,m\}$  
				$\left(\dfrac{2}{q_j}\right)_4=\left(\dfrac{q_j}{2}\right)_4$, then 
				$$E(K)=K(\sqrt{q_1},..., \sqrt{q_n}, \sqrt{\alpha_1},...,\sqrt{\alpha_m}). $$

				\item If  $m\geq1$ and $\exists j\{1,...,m\}$, $\left(\dfrac{2}{q_j}\right)_4\not=\left(\dfrac{q_j}{2}\right)_4$, then 
				$$E(K)=K(\sqrt{q_1},..., \sqrt{q_n},   \{\sqrt{\alpha_j^*} \}_{j\not= j_0}  ) .$$
				
			\end{enumerate}
			
			where $\alpha_j^*=\alpha_j$ if $\alpha_j \equiv  1 \pmod 4$ $($i.e. $\left(\dfrac{p}{q_j}\right)_4=\left(\dfrac{q_j}{p}\right)_4$$)$ and $\alpha_j^*=\alpha_{j_0}\alpha_j$ else.
			
			\item Assume that $\exists i\in \{1,...,n\}$  $q_i\equiv 3\pmod 4$, we have:
			\begin{enumerate}[\rm $\bullet$]
				\item If $m=0$, then
				$$E(K)=K(\{\sqrt{ q_i^*} \}_{i\not= i_0}).$$
				\item  If  $m\geq1$ and $\forall j  \in\{1,...,m\}$ $q_j\equiv 1\pmod 4$, we have 
				\begin{enumerate}[\rm $\star$]
					\item  If   $\forall j\in \{1,...,m\}$  
					$\left(\dfrac{2}{q_j}\right)_4=\left(\dfrac{q_j}{2}\right)_4$, then  
					$$E(K)=K(\{\sqrt{ q_i^*} \}_{i\not= i_0} ,    \sqrt{\alpha_1},..., \sqrt{\alpha_m}). $$ 
					\item   If   $\exists j\in \{1,...,m\}$  
					$\left(\dfrac{2}{q_j}\right)_4\not=\left(\dfrac{q_j}{2}\right)_4$, then  
					$$E(K)=K(\{\sqrt{ q_i^*} \}_{i\not= i_0} ,     \sqrt{\beta_1},..., \sqrt{\beta_m})  ) .$$ 
				\end{enumerate}
				\item   If  $m\geq1$ and $\exists j  \in\{1,...,m\}$ $q_j\equiv 3\pmod 4$, we have 
				$$E(K)=K(\{\sqrt{ q_i^*} \}_{i\not= i_0} ,     \sqrt{\beta_1},..., \sqrt{\beta_{m-1}})  ). $$ 
			\end{enumerate}
		\end{enumerate}
		
		Where $\beta_j$ is defined as follows:
		\begin{enumerate}[\indent\rm $\star$]
			\item  If $q_j\equiv 1\pmod 8$, we have 
			$$\beta_j	=\begin{cases}\alpha_j,& \text{ if } \left(\dfrac{2}{q_j}\right)_4=\left(\dfrac{q_j}{2}\right)_4;\\
				q_{i_0}\alpha_j, & \text{ if }\left(\dfrac{2}{q_j}\right)_4\not=\left(\dfrac{q_j}{2}\right)_4.\end{cases}$$
			
			\item If $q_j\equiv 7\pmod 8$, we have 
			$$\beta_j	=\begin{cases}\alpha_j,& \text{ if }  \alpha_j \equiv 1 \text{ or } (1+\sqrt{2})^2\pmod 4;\\
				q_{i_0}\alpha_j, & \text{ if } \alpha_j \equiv -1 \text{ or }-(1+\sqrt{2})^2\pmod 4.\end{cases}$$       
		\end{enumerate}
	\end{theorem}

	\begin{proof}
		We shall prove the last case. 	Inspired by this and  the proof of Theorem \ref{first thm},   the    reader can check the rest by using Lemma \ref{lemma 4.1} and \cite{aziziTmmZekhII}. 
		
		Assume that  $m\geq1$ and $\exists j  \in\{1,...,m\}$ $q_j\equiv 3\pmod 4$.
		By	Lemma \ref{lemma 4.1} and \cite[Theorems 5.3 and 5.4]{aziziTmmZekhII}, we have $r_2(\Delta(K)/{K^*}^2)=n+m-2$. Consider the following set
		$$\eta =\{\sqrt{ q_i^*} \}_{i\not= i_0}\cup \{ \sqrt{\beta_1},..., \sqrt{\beta_{m-1}}  \}.$$
		As in the proof of Theorem \ref{first thm}, we show that  $\eta$ is linearly independent modulo   $(K^*)^2$.

		\noindent$-$ Assume that  $j$ is such that $1\leq j \leq m$ and $q_j\equiv 1\pmod 8$. If $\left(\dfrac{2}{q_j}\right)_4=\left(\dfrac{q_j}{2}\right)_4$, then by Proposition \ref{lm symb buqia},   $x_j+y_j\equiv 1\pmod 4$.
		Thus,  $\beta_j=\alpha_j\equiv 1\pmod 4$.
		If $\left(\dfrac{2}{q_j}\right)_4\not=\left(\dfrac{q_j}{2}\right)_4$, then by Proposition \ref{lm symb buqia},    $x_j+y_j\equiv -1\pmod 4$. Thus,  $\beta_j=q_{i_0}\alpha_j\equiv 1\pmod 4$.
		
		It is easy seen that    ($\beta_j$) is the square of   a fractional ideal of $K$ for $1\leq j\leq m$. Then  by Proposition \ref{unramifed quad ex}, $K(\sqrt{\beta_j})/K$ is an unramified extension.

		\noindent$-$ Assume that $j$ is such that  $1\leq j \leq m$ and $q_j\equiv 7\pmod 8$.	It is clear that $x_j$ is odd.  Let us  show that $y_j$ is odd. We have,
		$2y_j^2\equiv x^2-q  \equiv2\pmod 8$. Therefore, one can check that $y_j$ is odd. 
		
		If $x_j\equiv y_j\equiv -1\pmod 4$, then we have
		$ \beta_j=- (1+\sqrt{2})(x_j+y_j\sqrt{2})= -(x_j+2y_j)-(x_j+y_j)\sqrt{2}$. Thus, 
		$\beta_j \equiv -1-2\sqrt{2}\equiv(3-2\sqrt{2})\equiv(1-\sqrt{2})^2 \pmod 4$.  
		If $x_j\equiv y_j\equiv 1\pmod 4$, then we have
		$ \alpha_j=(x_j+2y_j)+(x_j+y_j)\sqrt{2}\equiv 3+2\sqrt{2} \equiv (1+\sqrt{2})^2 \pmod 4$. Thus, 
		$\beta_j \equiv -1-2\sqrt{2}\equiv(3-2\sqrt{2})\equiv(1-\sqrt{2})^2 \pmod 4$.
		We similarly proceed for the other cases of $x_j$ and $y_j$ and show that the equation $\beta_j\equiv \xi^2\pmod 4$ has a solution.
		Since $q_j$ ramify in $K/k_0$, $(\beta_j)$ is the square of an ideal of $K$. Therefore, by Proposition \ref{unramifed quad ex}, $K(\sqrt{\beta_j})/K$ is an unramified extension.
	\end{proof}

	\begin{remark}
		In   all  cases  throughout this section    where $E(K)=K(\{\sqrt{ q_i^*} \}_{i\not= i_0}  )$ and $a=q_{i_0}$, we have $E(K)=K$.
	\end{remark}

\end{document}